\documentclass{amsart}

\usepackage{lmodern}
\usepackage[T1]{fontenc}
\usepackage[utf8]{inputenc}
\usepackage{hyperref}
\hypersetup{
  pdfauthor={Zijia Li, Hans-Peter Schröcker},
  pdftitle={Factorization Results for Left Polynomials in Some Associative Real Algebras: State of the Art, Applications, and Open Questions},
  pdfkeywords={},
  hidelinks,
}

\calclayout
\usepackage{amsmath,amssymb}
\usepackage{mathtools}
\usepackage{enumitem}

\usepackage{algorithm,algorithmic}

\newcommand{\C}{\mathbb{C}}
\newcommand{\D}{\mathbb{D}}
\renewcommand{\DH}{\mathbb{DH}}
\renewcommand{\H}{\mathbb{H}}
\newcommand{\HH}{\mathbb{S}}
\newcommand{\R}{\mathbb{R}}
\newcommand{\adef}{\leftarrow}
\newcommand{\tp}{\intercal}

\newcommand{\Clifford}[1]{C\kern-0.12em\ell_{(#1)}}

\newcommand{\evenClifford}[1]{C\kern-0.12em\ell^+_{(#1)}}
\newcommand{\Spin}[1]{\mathrm{Spin}_{(#1)}}
\newcommand{\hi}{\mathbf{i}_s}
\newcommand{\hj}{\mathbf{j}_s}
\newcommand{\hk}{\mathbf{k}_s}
\newcommand{\qi}{\mathbf{i}}
\newcommand{\qj}{\mathbf{j}}
\newcommand{\qk}{\mathbf{k}}
\newcommand{\eps}{\varepsilon}
\newcommand{\SO}[1][3]{\mathrm{SO}(#1)}
\newcommand{\SE}[1][3]{\mathrm{SE}(#1)}
\newcommand{\ci}{\mathrm{i}}
\newcommand{\Cj}[1]{{#1}^\ast}
\newcommand{\sandwich}[1]{\sigma_{#1}}
\newcommand{\lconcat}[2]{(#1,#2)}
\newcommand{\NQ}{\mathcal{N}}

\DeclareMathOperator{\mrpf}{mrpf}

\DeclareMathOperator{\lquo}{lquo}
\DeclareMathOperator{\lrem}{lrem}
\DeclareMathOperator{\rquo}{rquo}
\DeclareMathOperator{\rrem}{rrem}
\DeclareMathOperator{\czero}{czero}
\DeclareMathOperator{\gfactor}{gfactor}

\newtheorem{thm}{Theorem}
\newtheorem{lem}{Lemma}

\newtheorem{cor}{Corollary}
\theoremstyle{definition}
\newtheorem{defn}{Definition}
\theoremstyle{remark}
\newtheorem{rmk}{Remark}
\newtheorem{example}{Example}

\title[Factorization Results for Left Polynomials]{Factorization Results for Left Polynomials\\in Some Associative Real Algebras: State of the Art,\\Applications, and Open Questions}
\date{\today}

\author{Zijia Li}
\address[Zijia Li]{Institute of Discrete Mathematics and Geometry, Vienna University of Technology, Wiedner Hauptstrasse 8-10/104, 1040 Vienna, Austria}
\email{zijia.li@tuwien.ac.at}

\author{Daniel F. Scharler}

\author{Hans-Peter Schröcker}
\address[Daniel Scharler, Hans-Peter Schröcker]{Unit Geometry and CAD, University of Innsbruck, Technikerstr.~13, 6020 Innsbruck, Austria}
\urladdr{https://geometrie.uibk.ac.at/}
\email{\{daniel.scharler\}\{hans-peter.schroecker\}@uibk.ac.at}


\begin{document}

\begin{abstract}
    We discuss existence of factorizations with linear factors for (left)
  polynomials over certain associative real involutive algebras, most notably
  over Clifford algebras. Because of their relevance to kinematics and mechanism
  science, we put particular emphasis on factorization results for quaternion,
  dual quaternion and split quaternion polynomials. A general algorithm ensures
  existence of a factorization for generic polynomials over division rings but
  we also consider factorizations for non-division rings. We explain the current
  state of the art, present some new results and provide examples and counter
  examples.
 \end{abstract}

\maketitle

\section{Introduction}
\label{sec:introduction}

The factorization theory of polynomials over division rings has been developed
half a century ago in \cite{niven41,gordon65}. It gained new attention in recent
years because relations to mechanism science were unveiled
\cite{hegedus13,li15,li15darboux,li2016spatial,gallet16,li17,scharler17,li18,rad18}.
Quaternion polynomials parameterize rational spherical motions. For describing
motions in $\SE[2]$ or $\SE[3]$ dual quaternion polynomials are necessary. Their
factorization theory turned out to be more complicated and, arguably, more
interesting as well.

In this contribution we summarize the current state of the art in the
factorization theory of dual quaternion polynomials but we also demonstrate that
many results hold for polynomials over certain more general finite-dimensional
associative real algebras, most notably finite-dimensional Clifford algebras.
Throughout this paper we illustrate the general theory by three prototypical
examples with significantly different properties: The quaternions $\H$, the dual
quaternions $\DH$, and the split quaternions $\HH$ that can model planar
hyperbolic kinematics. A fundamental factorization algorithm, based on the
factorization of real polynomials, works for generic polynomials over these
algebras.

In Section~\ref{sec:polynomial-factorization} we recall some general results on
the factorization of polynomials over rings, in Section~\ref{sec:existence} we
present theoretical and algorithmic results (Theorem~\ref{th:factorization} and
Algorithm~\ref{alg:factorization}) on polynomial factorization over quaternions.
This is followed by some factorization examples that illustrate the intricacies
of polynomial factorization over skew rings
(Section~\ref{sec:factorization-examples}). There exist polynomials with no,
many or even infinitely many factorizations. Some of these factorizations can be
computed by means of Algorithm~\ref{alg:factorization} --- even if its general
applicability is limited to division algebras.
Section~\ref{sec:mechanism-science} explains relations of polynomial
factorization over quaternion rings to kinematics and mechanism science while
Section~\ref{sec:more-results} features a collection of known and new results
that allow to compute factorizations or to at least guarantee their existence.
The new results of this part include statements on factorizability of quadratic
split quaternion polynomials or unbounded motion polynomials.

\section{Polynomial Factorization over Rings}
\label{sec:polynomial-factorization}

Consider a possibly non-commutative ring $R$ and a polynomial $C = \sum_{i=0}^d
c_it^i$ in one indeterminate $t$ with coefficients $c_0$, $c_1$, \ldots, $c_d
\in R$. We define the product of two polynomials $A = \sum_{i=0}^d a_it^i$ and
$B = \sum_{i=0}^e b_it^i$ as
\begin{equation*}
  AB \coloneqq \sum_{i=0}^{d+e} c_i t^i
  \quad\text{where}\quad
  c_i = \sum_{j+k=i}a_jb_k.
\end{equation*}
This is really just one possible multiplication rule among others \cite{ore33}.
It is suitable for our purpose because in applications the indeterminate $t$
typically serves as a real parameter and $R$ is an associative real algebra.

We consistently write coefficients to the left of the indeterminate and hence
speak of \emph{left polynomials.} With addition defined in the usual way as $A +
B \coloneqq \sum_{i=0}^{\max\{d,e\}} (a_i+b_i)t^i$, the set $R[t]$ of left
polynomials in $t$ over $R$ is a ring. The \emph{evaluation $C(r)$} of $C$ at $r
\in R$ is defined as
\begin{equation*}
  C(r) \coloneqq \sum_{i=0}^d c_ir^i.
\end{equation*}
Besides this ``right evaluation'' there is also a ``left evaluation''
$\sum_{i=0}^d r^ic_i$ which gives rise to a completely symmetric theory. A ring
element $r$ is called a \emph{right zero} of $C$ if $C(r) = 0$ and a \emph{left
  zero} if its left evaluation at $r$ vanishes. Since left evaluation and
left zeros are not important for this paper, we introduce no special notation
for them. We will often simple speak of ``evaluation'' and ``zeros'' instead of
``right evaluation'' and ``right zeros''.

Evaluation of $C$ at a fixed value $r \in R$ is not generally a ring
homomorphism. For a counter example, take two non-commuting elements $r$, $q \in
R$ and set $C \coloneqq C_r C_q$ where $C_r \coloneqq t - r$ and $C_q \coloneqq
t - q$. We then have
\begin{equation*}
  C(r) = r^2 - (r + q)r + rq = rq - qr \neq 0
  \quad\text{but}\quad
  C_r(r) C_q(r) = 0
\end{equation*}
because $C_r(r) = 0$. However, we do have
\begin{equation*}
  C(q) = q^2 - (r + q)q + rq = 0.
\end{equation*}
This is no coincidence but consequence of Theorem~\ref{th:zero-factor} below.
Note that evaluation at $r$ is at least additive: For all $F$, $G \in R[t]$ we
have $(F+G)(r) = F(r) + G(r)$.

A polynomial $F$ is called a \emph{right factor} of $C$ if there exists a
polynomial $Q$ such that $C = QF$. Similarly, it is called a \emph{left factor}
if $C = FQ$. Polynomial division is possible in $R[t]$ but it is necessary to
distinguish between a left and a right version and to take into account
non-invertible coefficients.

\begin{thm}
  \label{th:polynomial-division}
  Given polynomials $F$, $G \in R[t]$ such that the leading coefficient of $G$
  is invertible, there exist unique polynomials $Q_\ell$, $Q_r$, $S_\ell$, and
  $S_r$ such that $\deg S_\ell < \deg G$, $\deg S_r < \deg G$ and $F = Q_\ell G
  + S_\ell = GQ_r + S_r$.
\end{thm}

\begin{defn}
  The polynomials $Q_\ell$, $Q_r$ in Theorem~\ref{th:polynomial-division} are
  called \emph{left} and \emph{right quotient,} respectively. The polynomials
  $S_\ell$ and $S_r$ are called \emph{left} and \emph{right remainder.} We
  denote them by $Q_\ell = \lquo(F,G)$, $Q_r = \rquo(F,G)$, $S_\ell =
  \lrem(F,G)$, and $S_r = \rrem(F,G)$, respectively.
\end{defn}

\begin{proof}[Proof of Theorem~\ref{th:polynomial-division}]
  Standard proofs for existence also work in this case. We do not repeat them
  here but instead refer to Algorithm~\ref{alg:euclidean}, the Euclidean
  Algorithm for left polynomial division. Its correctness is easy to see, the
  ``right'' version is explained in comments.

  As to uniqueness, assume that there are two left quotients and remainders,
  that is, $F = Q_1G + S_1 = Q_2G + S_2$. This implies
  \begin{equation*}
    (Q_1-Q_2)G = S_2 - S_1.
  \end{equation*}
  Now if $Q_1 \neq Q_2$, the polynomial on the left-hand side has degree greater
  than or equal to $\deg G$ because the leading coefficient of $G$ is
  invertible. But the degree on the right-hand side is strictly smaller. Hence
  $Q_1 = Q_2$ and also $S_1 = S_2$. In the same way we can prove uniqueness of
  right quotient and remainder.
\end{proof}

\begin{rmk}
  \label{rem:1}
  If the leading coefficient of $G$ fails to be invertible, neither existence
  nor uniqueness of quotient and remainder can be guaranteed. These phenomena
  will be illustrated in Examples~\ref{ex:1} and \ref{ex:2} below (after
  suitable associative real algebras will be constructed).
\end{rmk}

\begin{algorithm}
  \caption{Left Euclidean Algorithm}
  \label{alg:euclidean}
  \begin{algorithmic}
    \REQUIRE Polynomials $F$, $G \in R[t]$, leading coefficient of $G$ is
    invertible.
    \ENSURE Polynomials $Q$, $S \in R[t]$ such that $F = QG + S$ and $\deg S <
    \deg G$.
    \STATE $g \adef$ leading coefficient of $G$
    \STATE $F_0 \adef Fg^{-1}$, $G_0 \adef Gg^{-1}$ \qquad
    \COMMENT{Use $F_0 \adef g^{-1}F$, $G_0 \adef g^{-1}G$ for right division.}
    \STATE $Q \adef 0$, $S \adef F_0$
    \STATE $m \adef \deg S$, $n \adef \deg G_0$
    \WHILE{$m \ge n$}
        \STATE $r \adef$ leading coefficient of $S$
        \STATE $Q \adef Q + r t^{m-n}$, $S \adef S - rG_0t^{m-n}$ \qquad
        \COMMENT{Use $S \adef S - G_0rt^{m-n}$ for right division.}
        \STATE $m \adef \deg S$
    \ENDWHILE
    \RETURN $Q$, $Sg$ \qquad
    \COMMENT{Return $Q$, $gS$ for right division.}
  \end{algorithmic}
\end{algorithm}

The next result has been shown in \cite{gordon65} for division rings but it
holds true in more general rings (see \cite{hegedus13} for the case of dual
quaternions).

\begin{thm}
  \label{th:zero-factor}
  The ring element $r \in R$ is a right zero of $C$ if and only if $t-r$ is a
  right factor of~$C$.
\end{thm}

\begin{proof}
  Using polynomial division, we obtain $C = F + s$ where $F = Q(t-r)$ and $s \in
  R$. By uniqueness of polynomial division, $t-r$ is a right factor if and only
  if $s = 0$. Writing $Q = \sum_{i=0}^d q_it^i$, we compute
  \begin{equation*}
    Q(t-r)
    = \sum_{i=0}^d(q_it^i) (t-r)
    = \sum_{i=0}^dq_it^{i+1} - \sum_{i=0}^dq_irt^i 
  \end{equation*}
  whence
  \begin{equation*}
    F(r) = \sum_{i=0}^dq_ir^{i+1} - \sum_{i=0}^dq_irr^i = 0.
  \end{equation*}
  From $C(r) = F(r) + s = s$ we infer that $r$ is a right zero of $C$ if and only
  $s = 0$.
\end{proof}

Theorem~\ref{th:zero-factor} has a corollary which is sometimes useful:

\begin{cor}
  If $F$, $G \in R[t]$ and $S = \lrem(F,G)$, then $F(r) = S(r)$ for every zero
  $r$ of~$G$.
\end{cor}

\begin{proof}
  If $r$ is a zero of $G$, then $t-r$ is a right factor of $G$ and also of $QG$
  for $Q = \lquo(F, G)$. Hence, $F(h) = (QG)(h) + S(h) = 0 + S(h)$.
\end{proof}

\begin{defn}
  \label{def:factorization}
  We say that the polynomial $C \in R[t]$ of degree $n \ge 1$ \emph{admits a
    factorization} if there exist ring elements $c_n$, $h_1$, $h_2$, \ldots,
  $h_n$ such that $C = c_n(t-h_1)(t-h_2) \cdots (t-h_n)$.
\end{defn}

It will simplify things a lot if the leading coefficient $c_n$ of $C$ is
invertible. In this case, it is no loss of generality to assume $c_n = 1$
because $C$ admits a factorization if and only if $c_n^{-1}C$ does. We will
generally assume that $C$ is monic.

Theorem~\ref{th:zero-factor} relates zeros with linear right factors of $C$.
Using Theorem~\ref{th:polynomial-division} and Algorithm~\ref{alg:euclidean} it
is possible to compute linear right factors from zeros. This situation is
reminiscent of polynomial factorization over the complex numbers $\C$ but there
are fundamental differences due to non-commutativity and existence of
zero-divisors.

\section{Existence of Factorizations}
\label{sec:existence}

In the following, denote by $R$ a finite-dimensional associative real involutive
algebra with multiplicative identity $1$ and involution $\gamma$. The involution
$\gamma\colon R \to R$ has the following properties:
\begin{itemize}
\item $\gamma \circ \gamma$ is the identity on~$R$
\item $\forall a,b \in R,\ \alpha,\beta \in \R\colon \gamma(\alpha a + \beta b)
  = \alpha\gamma(a) + \beta\gamma(b)$
\item $\forall a,b \in R\colon \gamma(ab) = \gamma(b)\gamma(a)$
\end{itemize}
These properties already imply that $\gamma(1) = 1$:
\begin{equation*}
  \forall a \in R\colon 1 \cdot \gamma(a) = \gamma(a) \implies \gamma(1 \cdot \gamma(a)) = \gamma(\gamma(a)) \implies a \cdot \gamma(1) = a.
\end{equation*}
We are going to prove existence results of factorizations of left polynomials
over some $R$ for which the additional assumption $\gamma(a)a = a\gamma(a)$
holds for all $a \in R$. Theorem~\ref{th:factorization} below covers the case of
division rings (real numbers, complex numbers and quaternions by Frobenius'
Theorem) but in its formulation and proof we do not make direct use of
properties of these number systems. The reason is that the corresponding
Algorithm~\ref{alg:factorization} may make sense in the presence of
non-invertible elements as well. Variants and generalizations of
Theorem~\ref{th:factorization} and Algorithm~\ref{alg:factorization} with weaker
assumptions are given in Section~\ref{sec:more-results}.

Since the center of the ring $R$ contains $\R$, any polynomial $C \in R[t]$ has
a unique real monic factor of maximal degree. We denote this factor by $\mrpf C$
(the ``maximal real polynomial factor''). For reasons of simplicity, we assume
that it equals~$1$.

\begin{thm}
  \label{th:factorization}
  If a finite-dimensional associative real involutive algebra $R$ with
  involution $\gamma$ is
  \begin{enumerate}[label=\alph*)]
  \item a division ring and
  \item satisfies
    \begin{equation}
      \label{eq:1}
      \forall a \in R\colon \nu(a)\coloneqq \gamma(a)a = a\gamma(a) \in \R,
    \end{equation}
  \end{enumerate}
  then every monic polynomial $C \in R[t]$ of positive degree and with $\mrpf C
  = 1$ admits a factorization.
\end{thm}

By Frobenius' Theorem (see for example \cite{palais68}) $R$ is either the field
of real or complex numbers or the skew field of quaternions. Hence,
Theorem~\ref{th:factorization} does not present a new result. Moreover, the
involution $\gamma$ is the usual complex or quaternion conjugation and
requirement b) need not be stated as hypothesis. However, our formulation of
Theorem~\ref{th:factorization} already takes into account later generalizations
where condition a) will not be needed but condition b) will be crucial.

Let us drop for a moment the condition that $R$ is a division ring. If an
involution $\gamma$ as in Theorem~\ref{th:factorization} is given, the inverse
of $r \in R$ (if it exists) is $\gamma(r)/\nu(r)$. In particular, $r$ is
invertible if and only if $\nu(r) \neq 0$ and $\gamma(r)$ is unique up to sign.
If the involution $\gamma$ does not satisfy \eqref{eq:1}, we may instead
consider the multiplicative semigroup
\begin{equation}
  \label{eq:2}
  R^\gamma \coloneqq
  \{ a \in R \mid \gamma(a)a = a\gamma(a) \in \R \}.
\end{equation}
Examples for semigroups of this type are the pin and spin groups of Clifford
algebras. We may extend $\gamma$ to the involution
\begin{equation*}
  R[t] \to R[t],\quad
  \sum_{i=0}^n c_it^i \mapsto \sum_{i=0}^n \gamma(c_i)t^i
\end{equation*}
for polynomials over $R$. By abuse of notation, we denote it by the same symbol.
For $C \in R[t]$, condition \eqref{eq:1} implies that the \emph{norm polynomial}
$\nu(C) \coloneqq C\gamma(C) = \gamma(C)C$ of $C$ is in $\R[t]$. Also note that
we may perform the semigroup construction of Equation~\eqref{eq:2} for
polynomials:
\begin{equation}
  \label{eq:3}
  R^\gamma[t] \coloneqq
  \{ C \in R[t] \mid \gamma(C)C = C\gamma(C) \in \R[t] \}.
\end{equation}

\begin{rmk}
  \label{rem:2}
  In this context, our general assumption that the polynomial to be factorized
  is monic, is an actual restriction. If $\nu(C) \neq 0$, there is a suitable
  fractional linear parameter transformation $t \mapsto (\alpha t +
  \beta)(\gamma t + \delta)^{-1}$ with $\alpha$, $\beta$, $\gamma$, $\delta \in \R$
  and $\alpha\delta - \beta\gamma \neq 0$ that makes the leading coefficient $g$
  of $C' \coloneqq (\gamma t + \delta)^{\deg C}C$ invertible. Any factorization
  of $g^{-1}C'$ also gives rise to a factorization of $C$ and vice versa. If
  $\nu(C) = 0$, no suitable fractional linear parametrization exists.
  Polynomials with the property $\nu(C) = 0$ are not covered in this text. Their
  factorizability requires a separate investigation.
\end{rmk}

\begin{lem}
  \label{lem:2}
  Suppose that $R$, $C$, $\gamma$, and $\nu$ are as in
  Theorem~\ref{th:factorization} (but $R$ is not necessarily a division ring).
  If $M$ is a monic, quadratic and real factor of $\nu(C)$ and $S \coloneqq
  \lrem(C,M)$ satisfies $\nu(S) \neq 0$, then $S$ has a unique zero $h$ and
  $t-h$ is a right factor of~$C$.
\end{lem}

\begin{proof}
  Using polynomial division we can find $Q$, $S \in R[t]$ such that $C = QM + S$
  and $\deg S \le 1$. Moreover, because of
  \begin{equation*}
    \nu(C) = (QM+S)\gamma(QM+S)
           = (QM+S)(M\gamma(Q) + \gamma(S))
           = (\nu(Q)M + Q\gamma(S) + S\gamma(Q))M + \nu(S),
         \end{equation*}
  $M$ is also a factor of $\nu(S)$. Thus, there exists $c \in \R$ such that
  $\nu(S) = cM$. By assumption, $c \neq 0$ whence $S = s_1t + s_0$ with $s_0$,
  $s_1 \in R$ and $\nu(s_1) = c \neq 0$. Hence, there is a unique zero $h =
  -s_1^{-1}s_0$ of $S$ and $t-h$ is not only a right factor of $S$ but also
  of~$M$.
\end{proof}

\begin{proof}[Proof of Theorem~\ref{th:factorization}]
  We prove the theorem by induction on $n \coloneqq \deg C$. For $n = 1$ the
  statement is obvious. For the induction step, we pick a quadratic factor $M$
  of $\nu(C)$ and compute $S \coloneqq \lrem(C, M)$. The remainder polynomial
  always satisfies $\nu(S) \neq 0$ because $M$ cannot be a factor of $C$ and $R$
  is assumed to be a division ring. Hence, we may use Lemma~\ref{lem:2} to
  construct one right factor $t - h$. The induction hypothesis applied to
  $\lquo(C,t-h)$ then guarantees existence of a factorization.
\end{proof}

Our inductive proof of Theorem~\ref{th:factorization} gives rises to the
recursive Algorithm~\ref{alg:factorization} for computing factorizations of a
polynomial $C \in R[t]$. It has been used in \cite{hegedus13} to factor
quaternion and certain dual quaternion polynomials. If $M \in \R[t]$ is of
degree two, we denote the unique zero (according to Lemma~\ref{lem:2}) of
$\lrem(C,M)$ by $\czero(C,M)$. For two tuples $T_1$ and $T_2$ of polynomials we
denote by $(T_1,T_2)$ their concatenation.

\begin{algorithm}
  \caption{$\gfactor$: Factorization algorithm for polynomials based on
    Theorem~\ref{th:factorization}}
  \label{alg:factorization}
  \begin{algorithmic}
    \REQUIRE Monic polynomial $C \in R[t]$, $\deg C = n \ge 1$, $\mrpf C = 1$
    \ENSURE A tuple $(t-h_1,t-h_2,\ldots,t-h_n)$ of linear polynomials such that
    $C=(t-h_1)(t-h_2)\cdots(t-h_n)$.
    \IF{$\deg C = 0$}
    \RETURN $()$ \COMMENT{Empty tuple.}
    \ENDIF
    \STATE $M \adef$ quadratic, real factor of $\nu(C) \in \R[t]$
    \STATE $h \adef \czero(C, M)$
    \STATE $C \adef \rquo(C, t-h)$
    \RETURN $\lconcat{t-h}{\gfactor(C)}$
  \end{algorithmic}
\end{algorithm}

\begin{rmk}
  \label{rem:factorization}
  A few remarks on Algorithm~\ref{alg:factorization} are in order:
  \begin{itemize}
  \item Because in each recursion, a quadratic factor $M$ of the norm polynomial
    $\nu(C)$ is chosen, the algorithm is not deterministic. In fact, it
    generically gives rise to a finite number of different factorizations. The
    total number of factorizations depends on the number of irreducible (over
    $\R$) real quadratic factors of $\nu(C)$, the number of real linear factors
    of $\nu(C)$ and their respective multiplicities.
  \item Algorithm~\ref{alg:factorization} will produce all factorizations of
    $C$: If $C = C'(t-h)$, then $\nu(C) = \nu(C')\nu(t-h)$ and $\nu(t-h)$ is
    among the quadratic factors of $\nu(C)$.
  \end{itemize}
\end{rmk}

Also note that the assumption $\mrpf C = 1$ can be dropped for rings that
contain the complex numbers $\C$ as a subring. We may combine any factorization
$\mrpf C = (t-z_1)(t-z_2)\cdots(t-z_\ell)$ over (the subring isomorphic to) $\C$
with any factorization $\lquo(C,\mrpf C) = (t-h_1)(t-h_2)\cdots(t-h_m)$ to
obtain the factorization $C =
(t-z_1)(t-z_2)\cdots(t-z_\ell)(t-h_1)(t-h_2)\cdots(t-h_m)$.

Algorithm~\ref{alg:factorization} is based on a factorization of the real
polynomial $\nu(C)$ over $\R$. For moderate polynomial degrees, numeric
factorization of real polynomials is always possible
(\cite{qrgcd04,kaltofen06}), but the ensuing polynomial division may be tricky.
Without going into detail, we mention that it is possible to make
Algorithm~\ref{alg:factorization} numerically stable by using the
evaluation-interpolation univariate polynomial division algorithm (e.g., the
fast and robust algorithm based on the Fast Fourier Transform in
\cite{bini2012polynomial,pan2001gcd,li2010blind}) to compute $h \in R$ such that
$M$ divides $C(t-\gamma(h))$. Because of $M = (t-h)(t-\gamma(h))$, this is
equivalent to the computation in Algorithm~\ref{alg:factorization}.

\section{Factorization Examples}
\label{sec:factorization-examples}

In this section, we explicitly construct some rings over Clifford algebras and
present factorization examples for polynomials over those rings. Note that not
all polynomials in these examples satisfy the requirements of
Theorem~\ref{th:factorization} and Algorithm~\ref{alg:factorization}.
Nonetheless, it might be possible to use Algorithm~\ref{alg:factorization} to
compute factorizations.

\subsection{Clifford Algebras}
\label{sec:clifford-algebras}

Our brief introduction to Clifford algebras follows \cite{klawitter15} and
\cite[Section~9.1]{selig05}. In the real vector space $\R^n$ we consider a
quadratic form $\varrho\colon \R^n \to \R$. With respect to a basis
$(e_1,e_2,\ldots,e_n)$ it is described by a matrix $Q \in \R^{n \times n}$ via
$\varrho(e_i,e_j) = e_i^\tp \cdot Q \cdot e_j$. The defining relations for the
Clifford algebra are
\begin{equation}
  \label{eq:4}
  e_ie_j + e_je_i \coloneqq 2e_i^\tp \cdot Q \cdot e_j
  \quad\text{for all $i$, $j \in \{1,2,\ldots,n\}.$}
\end{equation}
With respect to a different basis, the same quadratic form is described by a
congruent matrix. Hence, by Sylvester's Theorem of Inertia, there is a basis
such that $Q$ is diagonal with the first $p$ diagonal entries equal to $1$, the
next $q$ diagonal entries equal to $-1$ and the remaining $r = n - p - q$
diagonal entries equal to $0$. We assume that this is the case for the chosen
basis $(e_1,e_2,\ldots,e_n)$, which, together with \eqref{eq:4}, implies $e_ie_j
= -e_je_i$ whenever $i \neq j$. For the product of successive basis elements we
also use the shorthand notation
\begin{equation*}
  e_{12 \ldots k} \coloneqq e_1e_2 \cdots e_k
  \quad\text{for $1 \le k \le n$}.
\end{equation*}
The span of all these element with multiplicative structure given by
\eqref{eq:4} is called a \emph{Clifford algebra} and will be denoted by
$\Clifford{p,q,r}$. An element of $\Clifford{p,q,r}$ can be written as
\begin{equation*}
  r = a_0 + \sum_{k=1}^n\;\sum_{i_1 < i_2 < \cdots < i_k} a_{i_1i_2\ldots i_k}e_{i_1i_2\ldots i_k}
\end{equation*}
where $a_0$, $a_1$, \ldots, $a_{12\ldots n} \in \R$ and all summation indices
are between $1$ and $n$. Often, the real unit $1$ is identified with an
additional basis element $e_0$ whence above sum starts with $r = a_0e_0 +
\ldots$ We will usually follow this convention.

The \emph{conjugation} $r \mapsto \Cj{r}$ in $\Clifford{p,q,r}$ is the
$\R$-linear anti-automorphism defined by
\begin{equation*}
  \Cj{(e_{i_1}e_{i_2} \cdots e_{i_k})} \coloneqq (-1)^k(e_{i_k} \cdots e_{i_2}e_{i_1}).
\end{equation*}
It gives rise to the \emph{norm $N(r) \coloneqq r\Cj{r}$.} Elements in the span
of $e_1$, $e_2$, \ldots, $e_n$ are called \emph{vectors} and we identify them
with elements of $\R^n$. The \emph{even sub-algebra} $\evenClifford{p,q,r}$ of
$\Clifford{p,q,r}$ is the sub-algebra generated by basis elements
$e_{i_1i_2\ldots i_k}$ with $k$ even (and by $e_0$). The \emph{spin group} is
\begin{equation*}
  \Spin{p,q,r} \coloneqq \{ r \in \evenClifford{p,q,r} \mid N(r) = \pm 1,\ \forall v \in \R^n\colon rv\Cj{r} \in \R^n\}.
\end{equation*}
The map $\sandwich{r}\colon v \mapsto rv\Cj{r}$ is called the \emph{sandwich
  operator.}

Clifford algebras comprise several well-known algebraic structures. In the
context of polynomial factorization, algebras that permit the construction of
isomorphisms to transformation groups of Euclidean and non-Euclidean spaces are
of special interest. There, factorization corresponds to the decomposition of
rational motions into products of elementary motions.

\subsubsection*{Quaternions}

An element of $\evenClifford{0,3,0}$ can be written as $r = a_0e_0 + a_1e_{12} +
a_2e_{13} + a_3e_{23}$. We have $e_{12}^2 = e_{12}e_{12} = -e_{12}e_{21} = -1$
and also $e_{13}^2 = e_{23}^2 = -1$. This even Clifford sub-algebra is
isomorphic to the quaternion algebra $\H$. The basis elements $e_{12}$,
$e_{13}$, and $e_{23}$ correspond, in that order, to the quaternion units $\qi$,
$\qj$, and $\qk$, respectively. We will usually use the quaternion notation and
write $r = a_0 + a_1\qi + a_2\qj + a_3\qk$. For $r$ as above, $N(r) = a_0^2 +
a_1^2 + a_2^2 + a_3^2 \ge 0$ and $\sandwich{r}(v) \in \R^3$ for all $v \in
\R^3$. Hence, the only defining condition for spin group elements is $N(r) = 1$.
The map $r \to \sandwich{r}$ is an isomorphism between
$\Spin{0,3,0}/\{\pm 1\}$ and $\SO$ and accounts for the importance of
$\Clifford{0,3,0}$ in spatial kinematics.

Also note that the factor group $\H^\times/\R^\times$ of the multiplicative
quaternion group modulo the multiplicative reals is isomorphic to $\SO$ via the
map that sends $r \in \H^\times$ to the map $x \in \R^3 \mapsto
\sandwich{r}(x)/N(r)$. This isomorphism is more useful in the context of
quaternion polynomial factorization ($C\Cj{C} = 1$ is only satisfied by the
constant polynomials $C = \pm 1$).

\subsubsection*{Split Quaternions}

Also kinematics in planar hyperbolic geometry may be treated by means of a
Clifford algebra. The construction is similar to the construction of $\H$ but is
based on the even Clifford algebra $\evenClifford{1,2,0}$. We set $\hi \coloneqq
e_{12}$, $\hj \coloneqq e_{13}$, $\hk \coloneqq e_{23}$ and denote the algebra
generated by $1$, $\hi$, $\hj$ and $\hk$ by $\HH$. The norm of $r = a_0 + a_1\hi
+ a_2\hj + a_3\hk \in \HH$ equals
\begin{equation*}
  N(r) = (a_0 + a_1\hi + a_2\hj + a_3\hk)(a_0 - a_1\hi - a_2\hj - a_3\hk) = a_0^2 - a_1^2 - a_2^2 + a_3^2.
\end{equation*}
We see that $N(\sandwich{r}(v)) = N(r)^2N(v)$ equals $N(v)$ for all vectors $v
\in \R^3$ if and only if $N(r) = \pm 1$. Hence $\Spin{1,2,0}$ is isomorphic to a
transformation subgroup of planar hyperbolic geometry. In contrast to the
quaternions $\H$, the norm of these so-called \emph{split quaternions} can
attain negative values. As in the case of quaternions we have $r^{-1} =
\Cj{r}/N(r)$ but the inverse element exists only if $N(r) \neq 0$. In
particular, $\HH$ is not a division ring and Theorem~\ref{th:factorization} is
not generally applicable.

\subsubsection*{Dual Quaternions}

An isomorphism from a Clifford algebra based group to the group $\SE$ of rigid
body displacements requires a more elaborate construction. An element of
$\evenClifford{3,0,1}$ is of the shape
\begin{equation*}
  r =
  a_0e_0 + a_3e_{12} - a_2e_{13} + b_1e_{14} +
  a_1e_{23} + b_2e_{24} + b_3e_{34} - b_0e_{1234}
\end{equation*}
with $a_0$, $a_1$, $a_2$, $a_3$, $b_0$, $b_1$, $b_2$, $b_3$ in $\R$. Its norm
equals
\begin{equation*}
  N(r) = (a_0^2 + a_1^2 + a_2^2 + a_3^2)e_0 - (a_0b_0 + a_1b_1 + a_2b_2 + a_3b_3)e_{1234}.
\end{equation*}
The spin group conditions are
\begin{equation*}
  a_0^2 + a_1^2 + a_2^2 + a_3^2 = 1,\quad
  a_0b_0 + a_1b_1 + a_2b_2 + a_3b_3 = 0
\end{equation*}
and the restriction of the conjugation map $r \mapsto \Cj{r}$ to $\Spin{3,0,1}$
(but not its extension to $\DH$) qualifies to play the role of $\gamma$ in
Theorem~\ref{th:factorization}.

The algebra of dual quaternions $\DH$ is obtained from $\H$ by extension of
scalars from the real numbers to the dual numbers $\D = \R[\eps]/\langle \eps^2
\rangle$. By Equation~(3.3) of \cite{klawitter15}, the map
\begin{multline*}
  a_0e_0 + a_3e_{12} - a_2e_{13} + b_1e_{14} +
  a_1e_{23} + b_2e_{24} + b_3e_{34} - b_0e_{1234}\\
  \mapsto
  a_0 + a_1\qi + a_2\qj + a_3\qk +
  \eps(b_0 + b_1\qi + b_2\qj + b_3\qk)
\end{multline*}
is an isomorphism between $\evenClifford{3,0,1}$ and the algebra $\DH$ of dual
quaternions. Again, we will prefer the dual quaternion notation in this text.
The spin group $\Spin{3,0,1}$ is isomorphic to $\SE$ by virtue of the action
$(x_1,x_2,x_3) \mapsto (y_1,y_2,y_3)$ where
\begin{equation*}
  1 + \eps(y_1\qi + y_2\qj + y_3\qk) = (a - \eps b)(1 + \eps(x_1\qi + x_2\qj + x_3\qk))(\Cj{a} + \eps\Cj{b})
\end{equation*}
and $r = a + \eps b \in \Spin{3,0,1}$. This is not quite the sandwich operator
but reduces to $\sandwich{a}$ for pure quaternions ($b = 0$). The translation
vector equals $a\Cj{b} - b\Cj{a}$. More generally, transformation groups of
arbitrary Euclidean spaces can be modeled by spin groups of Clifford algebras
\cite[Chapter~3]{klawitter15}.

Now that we have explicitly constructed several associative real algebras, we
are able to illustrate Remark~\ref{rem:1} on non-existence or non-uniqueness of
quotient and remainder by concrete examples:

\begin{example}
  \label{ex:1}
  Division of $F = t \in \DH[t]$ by $G = t\eps \in \DH[t]$ is not possible;
  quotient and remainder do not exist.
\end{example}
\begin{example}
  \label{ex:2}
  With
    \begin{alignat*}{2}
      Q_1 &= t + 1 + 3 \hi + \hj + 2 \hk,\quad&
      S_1 &= 1 + \hj,\\
      Q_2 &= t + 5 - \hi + \hj + 2 \hk,\quad&
      S_2 &= 1 - 3\hj - 4 \hk,
    \end{alignat*}
  and $G = (1 + \hi) t + 2 \hj - \hk$ we have
  \begin{equation*}
    F \coloneqq
    Q_1G + S_1 = Q_2G + S_2 = 
    (1 + \hi) t^2  + (4 + 4\hi + 5\hj + 2\hk)t + 5 - 5\hi + 6\hj - 7\hk.
  \end{equation*}
  Neither quotient nor remainder of the division of $F$ by $G$ are unique.
\end{example}

\subsection{Factorization examples}

We now illustrate some peculiarities of polynomial factorization over Clifford
algebras. We consider left polynomials over quaternions, split quaternions and
dual quaternions and demonstrate examples of typical and special factorizations.
Verifying correctness of the presented factorizations is straightforward. Often,
Algorithm~\ref{alg:factorization} could be used for computing factorizations,
even if not all requirements were fulfilled.

\begin{example}
  The polynomial $C = t^2 - (2\qi + \qj + 2)t + 2\qi + \qj + 2\qk + 1 \in \H[t]$
  admits the two factorizations
  \begin{equation*}
    C = (t - 2\qi - 1) (t - \qj - 1)
      = (t - \tfrac{4}{5}\qi + \tfrac{3}{5}\qj - 1)
        (t - \tfrac{6}{5}\qi - \tfrac{8}{5}\qj - 1).
  \end{equation*}
  Other factorizations do not exist. This is a generic case, factorizations can
  be computed by Algorithm~\ref{alg:factorization}.
\end{example}

\begin{example}
  The polynomial $C = t^3 - t^2 + t - 1 \in \R[t]$ admits the factorizations
  \begin{equation}
    \label{eq:5}
    C = (t-1)(t-h)(t-\Cj{h})
  \end{equation}
  where $h \in \mathbb{U}$ and
  \begin{equation*}
    \mathbb{U} \coloneqq \{ h \in \H \mid h^2 = -1 \} = \{ h_1\qi + h_2\qj + h_3\qk \mid h_1^2 + h_2^2 + h_3^2 = 1 \}.
  \end{equation*}
  All other factorizations are obtained by suitable permutations of the three
  factors in \eqref{eq:5}. These factors were found by factorizing $C$ over $\C$
  as $C = (t-1)(t-\ci)(t+\ci)$ and replacing the complex unit $\ci$ with $h$.
  Correctness of this construction follows from $h^2 = \ci^2 = -1$. As far as
  factorization of real polynomials is concerned, there is no essential
  algebraic difference between $h \in \mathbb{U}$ and~$\ci$.
\end{example}

The factorization theory of general quaternion polynomials is well understood
(see \cite{niven41}). Given $C \in \H[t]$, write $C = FG$ with $F = \mrpf C$. If
$F = \prod_\ell (t-t_\ell) \prod_m (t-z_m)(t-\overline{z}_m)$ with $t_\ell \in
\R$ and $z_m = x_m + \ci y_m \in \C$ is the factorization of $F$ over $\C$, all
factorizations over $\H$ are obtained by replacing $z_m = x_m + \ci y_m$ with
$x_m + h_my_m$ and $h_m \in \mathbb{U}$. All factorizations of $G$ are obtained
by Algorithm~\ref{alg:factorization} with different choices of the quadratic
factor $M$ at each recursion level. Depending on the number of different
quadratic factors (multiplicities of these factors), there exist between $1$ and
$(\deg G)!$ different factorizations of $G$. All factorizations of $C = FG$ are
obtained by combining factorizations of $F$ with factorizations of $G$ in an
obvious way.

\begin{example}
  \label{ex:3}
  The polynomial $C = t^2 - (2 + 2\hi + \hj)t + 2\hi + \hj + 2\hk + 1 \in \HH$
  admits precisely six different factorizations:
  \begin{equation*}
    \begin{aligned}
      C &= (t - \hj - 1)(t - 2\hi - 1),\\
        &= (t - \tfrac{6}{5}\hi - \tfrac{8}{5}\hj - 1)(t - \tfrac{4}{5}\hi + \tfrac{3}{5}\hj - 1),\\
        &= (t - \tfrac{3}{2}\hi + \tfrac{1}{2}\hj - \tfrac{3}{2}\hk + \tfrac{1}{2})(t - \tfrac{1}{2}\hi - \tfrac{3}{2}\hj + \tfrac{3}{2}\hk - \tfrac{5}{2}),\\
        &= (t - \tfrac{3}{2}\hi + \tfrac{1}{2}\hj + \tfrac{3}{2}\hk - \tfrac{5}{2})(t - \tfrac{1}{2}\hi - \tfrac{3}{2}\hj - \tfrac{3}{2}\hk + \tfrac{1}{2}),\\
        &= (t - \tfrac{1}{2}\hi - \tfrac{3}{2}\hj + \tfrac{1}{2}\hk - \tfrac{1}{2})(t - \tfrac{3}{2}\hi + \tfrac{1}{2}\hj - \tfrac{1}{2}\hk - \tfrac{3}{2}),\\
        &= (t - \tfrac{1}{2}\hi - \tfrac{3}{2}\hj - \tfrac{1}{2}\hk - \tfrac{3}{2})(t - \tfrac{3}{2}\hi + \tfrac{1}{2}\hj + \tfrac{1}{2}\hk - \tfrac{1}{2}).
    \end{aligned}
  \end{equation*}
  In spite of $\HH$ failing to be a division ring, above factorizations can be
  computed by means of Algorithm~\ref{alg:factorization}. The number of six
  factorizations is related to the fact that $\nu(C)$ is the product of
  \emph{four linear polynomials} $t$, $t+1$, $t-2$, and $t-3$. Hence, there
  exist six pairs $(M_1,M_2)$ of quadratic factors such that $\nu(C) = M_1M_2$:
  \begin{multline*}
    (M_1,M_2) \in \{ (t(t+1), (t-2)(t-3)),\
                     (t(t-2), (t+1)(t-3)),\
                     (t(t-3), (t+1)(t-2)),\\
                     ((t+1)(t-2), t(t-3)),\
                     ((t+1)(t-3), t(t-2)),\
                     ((t-2)(t-3), t(t+1)) \}.
  \end{multline*}
\end{example}

The sub-algebra $\langle 1, \hk \rangle$ is isomorphic to $\C$. Therefore, a
real polynomial can be factored over $\HH$ by replacing the complex unit $\ci$
with $\hk$. However, not all monic polynomials in $\HH[t]$ admit factorizations,
as the next example shows.

\begin{example}
  \label{ex:4}
  The polynomial $C = t^2 + 2\hi$ does not admit a factorization. This can be
  proved by means of Theorem~\ref{th:zero-factor}. Comparing coefficients on
  both sides of $C(x_0+x_1\hi+x_2\hj+x_3\hk) = 0$ we arrive at a system of
  algebraic equations in $x_0$, $x_1$, $x_2$, $x_3$ that has no real solutions.
  On the other hand, Algorithm~\ref{alg:factorization} gives $t^2 + 2\hk = (t -
  \hk + 1)(t + \hk - 1) = (t + \hk - 1)(t - \hk + 1)$.
\end{example}

As for polynomials in $\DH[t]$, even stranger examples exist:

\begin{example}
  \label{ex:no-factorization}
  The polynomial $C = t^2 + \eps \in \DH[t]$ admits no factorization. This can
  be shown in a similar way as in Example~\ref{ex:4}.
\end{example}

\begin{example}
  \label{ex:5}
  The polynomial $C = t^2 + 1 - \eps(\qj t - \qi) \in \DH$ has the infinitely
  many factorizations
  \begin{equation*}
    C = (t - \qk + \eps(a\qi + (b-1)\qj))(t + \qk - \eps(a\qi + b\qj))
    \quad\text{where}\quad a, b \in \R.
  \end{equation*}
\end{example}

\begin{example}
  \label{ex:9}
  The polynomial $C = t^2 + 1 + \eps\qi$ lies in the subset $\{ C \in \DH[t]
  \mid \nu(C) \in \R[t] \}$ of real norm polynomials but admits no factorization
  into linear factors belonging to this subset (compare also with
  Example~\ref{ex:multiplication-trick}). Nonetheless, it admits two
  two-parametric families of factorizations over $\DH$:
  \begin{equation}
    \begin{aligned}
      C &= (t+\qi+\eps(a\qj+b\qk-\tfrac{1}{2}))(t-\qi-\eps(a\qj+b\qk-\tfrac{1}{2})) \\
        &= (t-\qi+\eps(a\qj+b\qk+\tfrac{1}{2}))(t+\qi-\eps(a\qj+b\qk+\tfrac{1}{2}))
    \end{aligned}
  \end{equation}
  with arbitrary $a$, $b \in \R$.
\end{example}

\section{Application in Mechanism Science}
\label{sec:mechanism-science}

Factorization in (certain subsets of) Clifford algebras that are isomorphic to
transformation groups has important applications in kinematics and mechanism
science. The polynomial $C$ parameterizes a rational motion (all point
trajectories are rational curves), the factorization corresponds to the
decomposition of this motion into the product of ``elementary motions'' which
are parameterized by the linear factors of the form $t - h$.

In $\H$, $\HH$, and $\DH$ two elements $h$ and $\Cj{h}$ commute whence
\begin{equation}
  \label{eq:6}
  (t-h)(h-\Cj{h})(t-\Cj{h}) = (h-\Cj{h})(t^2 - (h+\Cj{h})t + h\Cj{h}) = (h-\Cj{h})\nu(t-h).
\end{equation}
This shows that $c \coloneqq h-\Cj{h}$ and $\sandwich{t-h}(c)$ are equal up to
multiplication with $\nu(t-h)$. This polynomial is real for quaternions and
split quaternions. For dual quaternions we add $\nu(t-h) \in \R[t]$ as an
assumption. Then $c$ is fixed under the spin group action of $t-h$ for any $t
\in \R$. In case of $\H$ or $\HH$, $c$ is a fix point of all displacements $t -
h$, $t \in \R$. Generically, it is the only fix point in $\H$ and one or one of
three fix points in $\HH$. From this, we may already infer that $t - h$
describes a rotation in spherical space or in the hyperbolic plane. In $\DH$,
the interpretation is similar but Equation~\eqref{eq:6} describes the action of
the displacement $t-h$ on the \emph{line with Pl\"ucker coordinate vector $c$.}
(More precisely, if $c = a + \eps b$, the line's Pl\"ucker coordinate vector
according to the convention of \cite{pottmann10} is $[a,-b]$.) The straight line
$c$ remains fixed and it is the axis of all spatial rotations described by $t-h$
for $t$ varying in~$\R$.

Hence, factorization of a polynomial $C$ in $\H$, $\HH$, or $\DH$ (with the
additional constraint $\nu(C) \in \R[t]$) corresponds to the decomposition of
the motion parameterized by $C$ into a sequence of coupled rotations
(translations in exceptional cases). Let us illustrate this with an example from
mechanism science.

The sub-algebra $\langle 1, \qi, \eps\qj, \eps\qk \rangle$ of $\DH$ modulo the
real multiplicative group $\R^\times$ is isomorphic to $\SE[2]$. A generic
quadratic polynomial $C$ in this sub-algebra admits two factorizations
\begin{equation*}
  C = (t-h_1)(t-h_2) = (t-k_1)(t-k_2)
\end{equation*}
(see Corollary~\ref{cor:generic-motion-polynomial} below). Each factorization
corresponds to the composition of two rotations and both compositions result in
the same motion. Hence, we may rigidly connect the centers of $h_1$, $h_2$,
$k_2$ and $k_1$ (in that order) to obtain a four-bar linkage. Its middle link
performs the motion parameterized by $C$. This is illustrated in
Figure~\ref{fig:anti-parallelogram}, left. It can be shown that the four-bar
linkage is an anti-parallelogram~\cite{gallet16}. A similar decomposition is not
possible for the polynomial of Example~\ref{ex:9}.

The same construction is possible in $\H$ and $\HH$ to obtain spherical and
hyperbolic anti-parallelogram linkages (four-bar linkages with equal opposite
sides) in the respective geometry. In case of $\HH$, it is necessary to use the
more general ``universal hyperbolic geometry'' in the sense of
\cite{wildberger13} in order to avoid awkward in-equality constraints.
Figure~\ref{fig:anti-parallelogram}, right, displays an example in the
Cayley-Klein model of hyperbolic geometry with absolute circle (or null
circle)~$N$. Note that this example admits precisely two factorizations and
gives rise to a unique four-bar linkage. The six factorizations of the
polynomial of Example~\ref{ex:3} give rise to a ``four-bar linkage'' with
\emph{six} possible legs. It cannot be visualized in traditional hyperbolic
geometry because all rotation centers lie in the exterior of $N$ but is
perfectly valid in universal hyperbolic geometry. A more detailed investigation
of the underlying geometry of these factorizations is given in \cite{li18c}.

\begin{figure}
  \centering
  \begin{minipage}{60mm}
    \centering
    \includegraphics{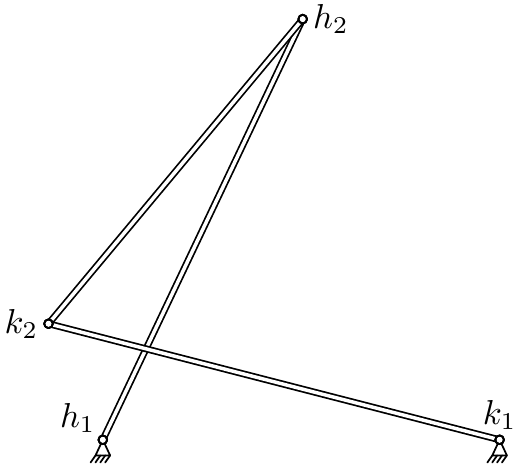}
  \end{minipage}  \begin{minipage}{80mm}
    \centering
    \includegraphics{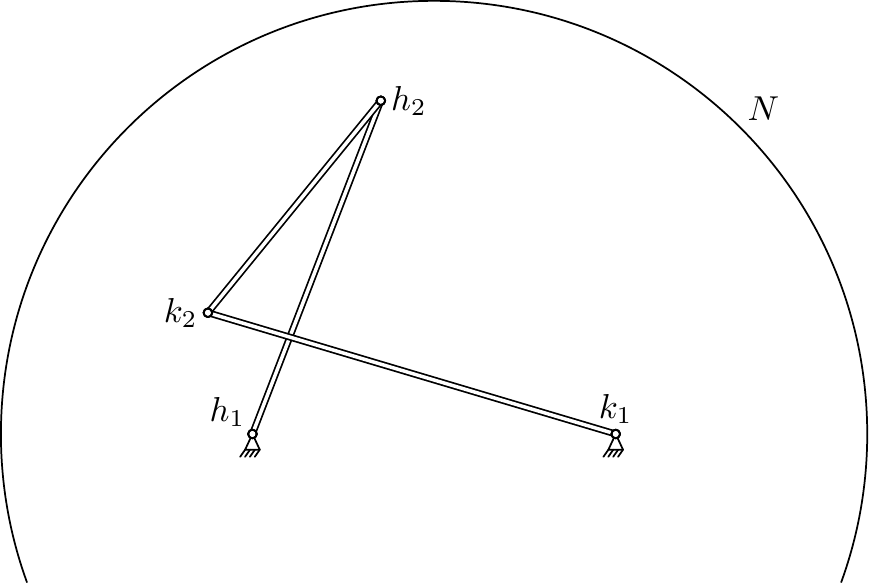}
  \end{minipage}  \caption{Anti-parallelogram mechanism in Euclidean geometry (left) and
    hyperbolic geometry (right)}
  \label{fig:anti-parallelogram}
\end{figure}

The polynomial of Example~\ref{ex:5} parameterizes a circular translation. This
motion can be generated by a parallelogram linkage
(Figure~\ref{fig:parallelogram}) which, indeed, admits infinitely many legs,
each corresponding to one of the infinitely many factorizations
\begin{equation*}
  C = (t-h_1)(t-h_2) = (t-k_1)(t-k_2) = (t-\ell_1)(t-\ell_2) = (t-m_1)(t-m_2) \ldots
\end{equation*}

\begin{figure}
  \centering
  \includegraphics{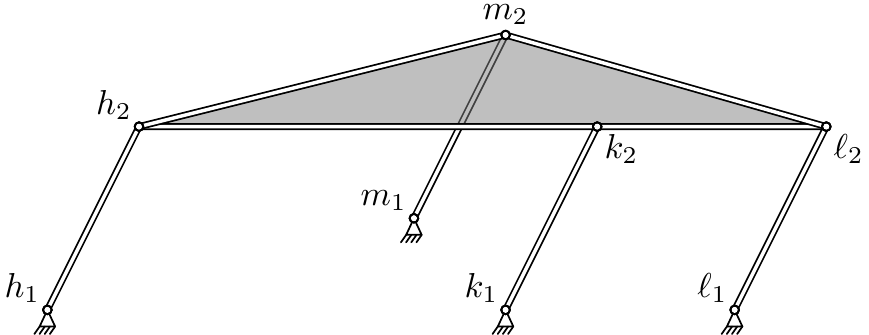}
  \caption{Parallelogram linkage}
  \label{fig:parallelogram}
\end{figure}

\begin{figure}
  \centering
  \includegraphics{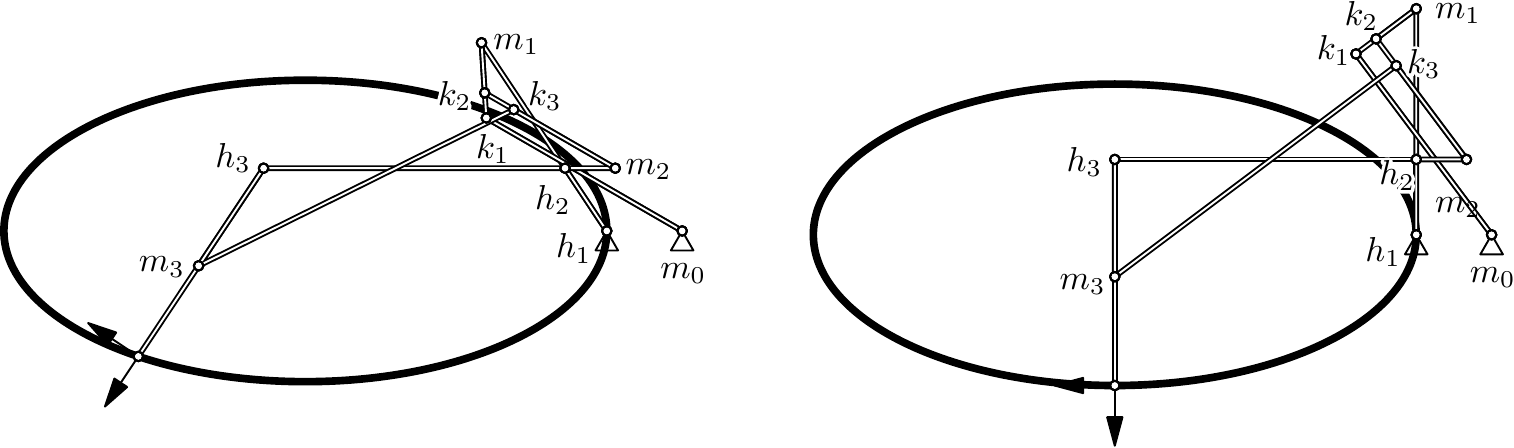}
  \caption{Scissor linkage to draw an ellipse.}
  \label{fig:scissor-linkage}
\end{figure}

The relevance of polynomial factorization in mechanism science goes beyond above
simple examples (see for example
\cite{hegedus15,li15darboux,li2016spatial,rad18}). It provides a more or less
automatic way to construct linkages from rational motions. One example related
to a rational version of Kempe's Universality Theorem is depicted in
Figure~\ref{fig:scissor-linkage}. Any rational planar or spatial curve (an
ellipse in Figure~\ref{fig:scissor-linkage}) can be drawn by a scissor like
linkage whose number of joints is linear in the curve degree
\cite{gallet16,li17}.

\section{More Factorization Results and Examples}
\label{sec:more-results}

It is unsatisfactory that Theorem~\ref{th:factorization} and
Algorithm~\ref{alg:factorization} are limited to division rings only. A detailed
inspection of the proof of Theorem~\ref{th:factorization} shows that the
property of $R$ being a division ring guarantees existence of an unique zero of
the remainder polynomial $S$. However, as already demonstrated,
Algorithm~\ref{alg:factorization} may work in more general circumstances and
even if it fails, factorizations may exists. In this section, we present
miscellaneous existence and non-existence results for factorizations of
polynomials in a certain subset of a finite-dimensional associative real
involutive algebra whose involution $\gamma$ does not generally satisfy
Equation~\eqref{eq:1}, but $\gamma$ restricted to this subset does.

Note that Clifford algebras allow simple constructions of subsets that fall into
this category. We define a suitable involution $\gamma$ by defining
$\gamma(e_\ell) = e_\ell$ or $\gamma(e_\ell) = -e_\ell$ for $\ell \in
\{0,1,\ldots,n\}$ and then extended $\gamma$ to the complete Clifford algebra by
$\R$-linearity and the property $\gamma(ab) = \gamma(b)\gamma(a)$. The subset
$R^\gamma$ defined in Equation~\eqref{eq:2} then satisfies all general
assumptions of this section.

\subsection{Applicability of Algorithm~\ref{alg:factorization}}
\label{sec:applicability}

An obvious pre-requisite for Algorithm~\ref{alg:factorization} is that $\nu(C)$
is a non-zero real polynomial. The non-vanishing of $\nu(C)$ is guaranteed by
our restriction to monic polynomials (compare Remark~\ref{rem:2}). The crucial
property that then ensures applicability of (one iteration of)
Algorithm~\ref{alg:factorization} is that $S = \lrem(C,M)$ has a unique zero
($\czero(C,M)$ is well-defined). If this is the case,
Algorithm~\ref{alg:factorization} produces polynomials $C'$, $t-h$ \emph{which
  are again in $R^\gamma[t]$} (even if $S$ is not): By construction
$(t-h)\gamma(t-h) = M \in \R[t]$ whence $\nu(C')$ must be real as well. In
particular, $C'$ is suitable as input for a further iteration of
Algorithm~\ref{alg:factorization}. In order to have a convenient notion for the
vanishing of the remainder polynomial, we state the following definition.

\begin{defn}
  \label{def:pseudofactor}
  Given two polynomials $F$, $G \in R[t]$ where the leading coefficient of $G$
  is invertible, $G$ is called a \emph{left pseudofactor} of $F$, if $\rrem(F,
  G)$ has vanishing norm and a \emph{right pseudofactor} of $F$, if $\lrem(F, G)$
  has vanishing norm.
\end{defn}

Obviously, left and right factors are also left and right pseudofactors,
respectively. If a left pseudofactor is real then it is also a right
pseudofactor and vice versa. In this case we simply speak of a
\emph{pseudofactor}. With the help of the involution $\gamma$, real
pseudofactors can be found by factorizing $\nu(C)$:

\begin{thm}
  \label{thm:real-pseudofactors}
  A real pseudofactor of $C$ is a factor of $\nu(C)$.
\end{thm}
\begin{proof}
  If $M$ is a real pseudofactor of $C$, there exist $Q$, $S \in R[t]$ with $C =
  QM + S$, $\deg S < \deg M$, and $\nu(S) = 0$. But then
  \begin{multline*}
    \nu(C) = C\gamma(C) = (QM+S)\gamma(QM+S) \\
    = \nu(Q)M^2 + (Q\gamma(S)+S\gamma(Q))M + \underbrace{\nu(S)}_{=0}
    = (\nu(Q)M + Q\gamma(S) + S\gamma(Q))M
  \end{multline*}
  and $M$ is indeed a factor of~$\nu(C)$.
\end{proof}

As shown in \cite{hegedus13}, Algorithm~\ref{alg:factorization} works for
polynomials in an important subsemigroup of~$\DH[t]$ that have no real pseudofactors.

\begin{defn}
  A polynomial $C = P + \eps Q \in \DH[t]$ with $P$, $Q \in \H[t]$ is called a
  \emph{motion polynomial} if $\nu(C) \in \R[t]$ and $\nu(C) \neq 0$. It is
  called \emph{generic} if $\mrpf P = 1$.
\end{defn}

Motion polynomials form a subsemigroup of a special instance of the semigroup
constructed in Equation~\eqref{eq:3}. Hence, we may at least try to factor
motion polynomials by means of Algorithm~\ref{alg:factorization}. For generic
motion polynomials, which are exactly those motion polynomials that do not have
real pseudofactors, it is guaranteed to work:

\begin{lem}
  \label{lem:3}
  Let $C = P + \eps Q$ be a monic motion polynomial. If $M$ is a monic real
  quadratic factor of $\nu(C)$ but not a pseudofactor of $C$, then $\czero(C,M)$
  is well defined (that is, $S \coloneqq \lrem(C,M)$ has a unique zero).
\end{lem}

\begin{proof}
  Because $M$ is not a pseudofactor of $C$ we have $\nu(S) \neq 0$. Similar to
  the proof of Lemma~\ref{lem:2} we conclude that there exist \emph{dual number}
  $c \in \D$ such that $\nu(S) = cM$ and $c \neq 0$. We claim that $c$ is
  invertible (in contrast to the proof of Lemma~\ref{lem:2} this is not implied
  b $c \neq 0$). Assume that $c$ is not invertible, that is, $c \in \eps\R$.
  This implies that $\nu(S) \in \eps\R[t]$ whence $M$ divides the primal part of
  $C$. But this is not possible because $M$ is assumed to be no pseudofactor of
  $C$. Hence $c \neq 0$ is invertible. It equals the norm of the leading
  coefficient of $S$ and this coefficient is invertible by a well-known property
  of dual numbers. Therefore, there exists a unique dual quaternion zero $h$
  of~$S$.
\end{proof}

\begin{cor}[\cite{hegedus13}]
  \label{cor:generic-motion-polynomial}
  A generic motion polynomial $C = P + \eps Q \in \DH[t]$ with $P$, $Q \in
  \H[t]$ admits a factorization.
\end{cor}

\begin{proof}
	We recap the inductive proof of Theorem~\ref{th:factorization} and show that
  the necessary conclusion can be obtained for a generic motion polynomial $C$.
  Again the base case is obvious. Consider a quadratic real factor $M$ of
  $\nu(C)$. By assumption it is not a pseudofactor of $C$ whence $h \coloneqq
  \czero(C,M)$ is well defined by Lemma~\ref{lem:3}. The considerations at the
  beginning of Section~\ref{sec:applicability} show that $t-h$ and $C' \coloneqq
  \lquo(C,t-h)$ are motion polynomials. If $C'$ was not generic, the same would
  be true for $C$, contrary to our assumption. Hence the induction hypothesis
  can be applied to $C'$ and the proof is complete.
\end{proof}

\begin{rmk}
  Algorithm~\ref{alg:factorization} can be used to factor motion polynomials as
  long as $\czero(C,M)$ is well-defined. In the generic case this is guaranteed.
\end{rmk}

Factorization results for non-generic motion polynomials and non-motion
polynomials will be discussed later in Sections~\ref{sec:non-generic}
and~\ref{sec:projection}, respectively. We conclude this section with an example
to demonstrate that success of Algorithm~\ref{alg:factorization} for a split
quaternion polynomial $C$ may depend on the order of quadratic factors of
$\nu(C)$.

\begin{example}
  \label{ex:10}
  The polynomial
  \begin{multline*}
    C = t^4 - (\hi - 3\hj + 2\hk + 9)t^3 + (7\hi - 12\hj + 33\hk + 43)t^2\\
    - (82\hi - 59\hj + 146\hk + 38)t + 162\hi - 188\hj + 213\hk - 103
  \end{multline*}
  admits the factorization $C = (t-h_1)(t-h_2)(t-h_3)(t-h_4)$ where
  \begin{gather*}
    h_1 = 3\hi+\tfrac{21}{2}\hj+\tfrac{23}{2}\hk+2,\quad
    h_2 = -\tfrac{91}{51}\hi-\tfrac{2151}{221}\hj-\tfrac{6791}{663}\hk+2,\\
    h_3 = \tfrac{91}{51}\hi-\tfrac{2667}{884}\hj-\tfrac{6649}{2652}\hk+2,\quad
    h_4 = -2\hi-\tfrac{3}{4}\hj+\tfrac{13}{4}\hk+3
  \end{gather*}
  This factorization can be computed by Algorithm~\ref{alg:factorization}. With
  \begin{equation*}
    M_1 \coloneqq t^2 - 6t + 15,\quad
    M_2 \coloneqq t^2 - 4t - 2,\quad
    M_3 \coloneqq t^2 - 4t + 11,\quad
    M_4 \coloneqq t^2 - 4t + 17
  \end{equation*}
  we have $\nu(C) = M_1M_2M_3M_4$ and
  \begin{equation*}
    \begin{aligned}
      h_4 &= \czero(C,M_1),\\
      h_3 &= \czero(C', M_2)\quad\text{where}\quad C' = \lquo(C,t-h_4),\\
      h_2 &= \czero(C'',M_3)\quad\text{where}\quad C'' = \lquo(C',t-h_3),\\
      h_1 &= t - \lquo(C'',t-h_2).
    \end{aligned}
  \end{equation*}
  A different order of quadratic factors may not work. With $k_4 = -\hi - \hj +
  3\hk + 2 = \czero(C,M_3)$ we have
  \begin{multline*}
    C' \coloneqq \lquo(C, t-k_4)\\
    = t^3-(2\hi-2\hj-\hk+7)t^2+(17\hi+4\hj+10\hk+26)t-52\hi-20\hj-35\hk-37
  \end{multline*}
  but
  \begin{equation*}
    S \coloneqq \lrem(C',M_1) = (5\hi+16\hj+16\hk+5)t-22\hi-50\hj-50\hk-22
  \end{equation*}
  and $\nu(S) = 0$. Thus $M_1$ is a pseudofactor of $C'$ but not of $C$.
  Algorithm~\ref{alg:factorization} with this particular ordering of quadratic
  factors of $\nu(C)$ does not work.
\end{example}

\subsection{Factorization of Quadratic Split Quaternion Polynomials}

As demonstrated in Example~\ref{ex:4}, not all monic polynomials in $\HH[t]$
admit factorizations. Here, we present a sufficient criterion for
factorizability of quadratic polynomials in $\HH[t]$. It relates existence of
factorizations with the geometry of the projective space $P(\HH)$ over the
vector space $\HH$. Given a split quaternion $x \in \HH$ we denote the
corresponding point in $P(\HH)$ by $[x]$. Projective span is denoted by the
symbol ``$\vee$''.

\begin{defn}
  The quadric $\NQ$ in $P(\HH)$ given by the bilinear form $q\colon \HH \times
  \HH \to \R$, $(x,y) \mapsto x\Cj{y} + y\Cj{x}$ is called the \emph{null
    quadric.} A straight line contained in $\NQ$ is called a \emph{null line.}
\end{defn}

A point $[x]$ lies on the null quadric $\NQ$ if and only if $\nu(x)$ vanishes. It
is easy to see (Lemma~\ref{lem:left-right-ruling} below) that $\NQ$ is of
hyperbolic type and contains two families of lines. In particular, null lines do
exist.

\begin{thm}
  \label{th:quadratic-split}
  A quadratic polynomial $C = c_2t^2 + c_1t + c_0 \in \HH[t]$ with invertible
  leading coefficient $c_2$ admits a factorization if the vectors $c_0$, $c_1$
  and $c_2$ are linearly independent.
\end{thm}

\begin{lem}
  \label{lem:line-quadric}
  The linear polynomial $S = s_1t + s_0 \in \HH[t]$ with linearly independent
  coefficients $s_0$ and $s_1$ satisfies $S\Cj{S} = 0$ if and only if the
  straight line $[s_0] \vee [s_1]$ is a null line.
\end{lem}

\begin{proof}
  Because of $S\Cj{S} = s_1\Cj{s}_1t^2 + (s_1\Cj{s_0}+s_0\Cj{s_1})t +
  s_0\Cj{s_0}$ we have $S\Cj{S} = 0$ if and only if $q(s_0,s_0) = q(s_0,s_1) =
  q(s_1,s_1) = 0$. This is precisely the condition for the straight line $[s_0]
  \vee [s_1]$ to be contained in the quadric~$\NQ$.
\end{proof}

\begin{lem}
  \label{lem:left-right-ruling}
  The quadric $\NQ$ contains two families of lines (the \emph{left} and the
  \emph{right family}) which are distinguished by the following property: For
  any two points $[p_1]$, $[q_1]$ on a line of the left family, there exists
  $r_1 \in \HH$ such that $q_1 = r_1p_1$. For any two points $[p_2]$, $[q_2]$ on
  a line of the right family, there exists $r_2 \in \HH$ such that $q_2 =
  p_2r_2$.
\end{lem}

\begin{proof}
  With $x = x_0 + x_1\hi + x_2\hj + x_3\hk$ we have $\frac{1}{2}q(x,x) = x_0^2 -
  x_1^2 - x_2^2 + x_3^2$. Hence, the quadric $\NQ$ is of hyperbolic type and,
  indeed, carries two families of rulings. These are given as $[a] \vee [b]$
  where
  \begin{equation}
    \label{eq:7}
    a = 1 + \cos\varphi \hi + \sin\varphi \hj,\quad
    b = -\sin\varphi \hi + \cos\varphi \hj + e \hk,
  \end{equation}
  $\varphi \in [0,2\pi)$ and $e = 1$ or $e = -1$. Any point on $[c] \in [a] \vee [b]$ can be written as
  $c = \alpha a + \beta b$ and it suffices to discuss solvability of the
  equations $ax = c$ and $xa = c$. Since both equations are linear in the
  coefficients of $x$, a straight-forward calculation yields the solution
  \begin{equation*}
    x = (\alpha - x_1\cos\varphi - x_2\sin\varphi) + x_1\hi + x_2\hj + (\beta + x_1\sin\varphi - x_2\cos\varphi)\hk
  \end{equation*}
  of $\{ e = 1, xa = c\}$ and the solution
  \begin{equation*}
    x = (\alpha - x_1\cos\varphi - x_2\sin\varphi) + x_1\hi + x_2\hj - (\beta + x_1\sin\varphi - x_2\cos\varphi)\hk
  \end{equation*}
  of $\{ e = -1, ax = c \}$. The systems $\{ e = 1, xa = c \}$ and $\{ e = -1, ax
  = c \}$ have no solution.
\end{proof}

\begin{proof}[Proof of Theorem~\ref{th:quadratic-split}]
  As usual, it is sufficient to prove the statement for monic polynomials, that
  is, $c_2 = 1$. We pick a monic quadratic factor $M_1$ of $\nu(C)$. The
  remainder polynomial $S_1\coloneqq \lrem(C,M_1)$ is (at most) of degree one
  and we can write $S_1 = s_1t + s_0$. If $\nu(S_1) \neq 0$, $h \coloneqq
  \czero(C, M_1)$ is well defined by Lemma~\ref{lem:2}. This means that one
  iteration of Algorithm~\ref{alg:factorization} can be applied to $C$ to obtain
  a right factor $t-h$ and consequently a factorization, even if $\HH$ is not a
  division ring. (The division ring property in the proof of
  Theorem~\ref{th:factorization} allows to conclude $\nu(S_1) \neq 0$ which is
  an assumption at this point.)

  Next we consider the remaining case $\nu(S_1) = 0$. If the coefficients $s_1$
  and $s_0$ are linearly dependent, the coefficients of $C$ are linearly
  dependent as well. Hence, $s_1$ and $s_0$ are linearly independent and we may
  assume that $S_1$ parameterizes a null line. With $M_2 \coloneqq M_1 + S_1 +
  \Cj{S_1}$ and $S_2 \coloneqq -\Cj{S_1}$ we have $\nu(C) = M_1M_2$ and $C = M_1
  + S_1 = M_2 + S_2$. By Lemma~\ref{lem:left-right-ruling}, $S_1$ or $S_2$ have
  a right zero. Lets assume, without loss of generality, that $S_1$ has this
  property. By a parameter transformation $t \mapsto t + u$ with a suitable $u
  \in \R$ we can ensure that $M_1$ is of the form $M_1 = t^2 + m$ with $m \in
  \R$. We have to show that there is a common right zero of $M_1$ and $S_1$. The
  right zeros of $S_1$ can be computed similarly as in the proof of
  Lemma~\ref{lem:left-right-ruling}: There exist $\alpha_0$, $\beta_0$,
  $\alpha_1$, $\beta_1 \in \R$ such that
  \begin{equation*}
    s_0 = \alpha_0 a + \beta_0 b\quad\text{and}\quad
    s_1 = \alpha_1 a + \beta_1 b
  \end{equation*}
  with $a$, $b$ as in \eqref{eq:7} with $e = -1$. With $h = h_0 + h_1\qi +
  h_2\qj + h_3\qk$, the solution to $s_1h + s_0 = 0$ is given by
  \begin{equation}
    \label{eq:8}
    \begin{aligned}
    h_0 &= \frac{-1}{\alpha_1^2+\beta_1^2}(
        \alpha_0\alpha_1+\beta_0\beta_1
        +((\alpha_1^2-\beta_1^2)\cos\varphi - 2\alpha_1\beta_1 \sin\varphi)h_1
        +(2\alpha_1\beta_1\cos\varphi + (\alpha_1^2-\beta_1^2)\sin\varphi)h_2
    ),\\
    h_3 &= \frac{1}{\alpha_1^2+\beta_1^2}
    (
        \alpha_1\beta_0 -\alpha_0\beta_1
        -(2\alpha_1\beta_1 \cos\varphi + (\alpha_1^2 - \beta_1^2)\sin\varphi)h_1
        +((\alpha_1^2 - \beta_1^2)\cos\varphi - 2\alpha_1\beta_1\sin\varphi)h_2
    )
    \end{aligned}
  \end{equation}
  with arbitrary real numbers $h_1$, $h_2$. A straightforward calculation shows
  that there is precisely one right zero of $M_1$ in this solution set. It is
  given by
  \begin{equation}
    \label{eq:99}
    h = \frac{1}{2(\alpha_0\beta_1-\alpha_1\beta_0)}(h_1\qi + h_2\qj + h_3\qk)
  \end{equation}
  where
  \begin{equation*}
    \begin{aligned}
    h_1 &= ((\alpha_1^2-\beta_1^2)m+\alpha_0^2 - \beta_0^2)\sin\varphi+2(\alpha_1\beta_1m + \alpha_0\beta_0)\cos\varphi,\\
    h_2 &= ((\beta_1^2-\alpha_1^2)m-\alpha_0^2+\beta_0^2)\cos\varphi+2(\alpha_1\beta_1m+\alpha_0\beta_0)\sin\varphi,\\
    h_3 &= -\alpha_0^2-\beta_0^2 - (\alpha_1^2+\beta^2_1)m.
    \end{aligned}
  \end{equation*}
  Note that the denominator of \eqref{eq:99} does not vanish because otherwise
  the coefficients of $S_1$ and consequently also the coefficients of $C$ would
  be linearly dependent.

  The quaternion $h$ is a common right zero of $M_1$ and $S_1$. By
  Theorem~\ref{th:zero-factor}, $t-h$ is a right factor of $M_1$ and $S_1$ and
  hence also of $C = M_1 + S_1$. This implies existence of a factorization.
  
\end{proof}

\begin{example}
  We illustrate the ``interesting'' case in the proof of
  Theorem~\ref{th:quadratic-split} by an example. Consider the polynomial $C =
  t^2 + (1+\hi)t + 1 + \hj - \hk \in \HH[t]$. We have $C\Cj{C} = M_1M_2$ with
  $M_1 = t^2 + 1$ and $M_2 = (t + 1)^2$. Polynomial division yields $C = M_1 +
  S_1 = M_2 + S_2$ with
  \begin{equation*}
    S_1 = (1+\hi)t + \hj - \hk
    \quad\text{and}\quad
    S_2 = (-1+\hi)t + \hj - \hk.
  \end{equation*}
  Note that $\nu(S_1) = \nu(S_2) = 0$. The remainder $S_2$ has no right zeros,
  while the right zeros of $S_1$ are of the form
  \begin{equation*}
    h = -h_1 + h_1\hi + h_2\hj + (1 + h_2)\hk,
    \quad h_1, h_2 \in \R.
  \end{equation*}
  The unique right zero $h = \hk$ of $M_1$ among these solutions is obtained for
  $h_1 = h_2 = 0$. Indeed, we have the factorization $C = (t + 1 + \hi + \hk)(t
  - \hk)$.
\end{example}

\subsection{Factorization of Non-Generic Motion Polynomials}
\label{sec:non-generic}

We have already mentioned (and proved) the result of \cite{hegedus13} on
existence of factorizations of generic motion polynomials. These are polynomials
$C = P + \eps Q \in \DH[t]$ with $P$, $Q \in \H[t]$ such that $\mrpf P = 1$ and
$\nu(C) \neq 0$. If $\mrpf P \neq 1$, general criteria on existence of
factorizations are difficult to formulate. However, we would like to mention
recent results by \cite{li17,li18} that ensure existence of factorizations for
suitable multiples of not necessarily generic but bounded motion polynomials.

\begin{defn}
  A motion polynomial $C = P + \eps Q$ with $P$, $Q \in \H[t]$ is called
  \emph{bounded} if $\mrpf P$ has no real zeros and \emph{unbounded} otherwise.
\end{defn}

The name ``bounded'' comes from the fact that all trajectories of a bounded
motion polynomials are bounded rational curves.

\begin{thm}[\cite{li17,li18}]
  Consider a bounded monic motion polynomial $C = P + \eps Q \in \DH[t]$ with
  $P$, $Q \in \H[t]$.
  \begin{itemize}
  \item There exists a polynomial $S \in \R[t]$ of degree $\deg S \le \deg\mrpf
    P$ such that $CS$ admits a factorization.
  \item If $\gcd(P,\nu(Q)) = 1$ there exists a polynomial $D \in \H[t]$ of
    degree $\deg D = \frac{1}{2}\deg\mrpf P$ such that $CD$ admits a
    factorization.
  \end{itemize}
\end{thm}

The algorithm of \cite{li18} for computing the co-factor $D$ is too complicated
to be discussed here. We confine ourselves to a simple example and remark that
some aspects of this factorization algorithm are used in our proof of
Theorem~\ref{th:unbounded} below.

\begin{example}
  \label{ex:multiplication-trick}
  Consider the polynomial $C = t^2 + 1 + \eps\qi$. As mentioned in
  Example~\ref{ex:9}, it admits no factorization with motion polynomial factors.
  But with $S = t^2+1$ and $D = t - \qk$ we have
  \begin{gather*}
    CS = 
    (t+\tfrac{3}{5}\qj-\tfrac{4}{5}\qk)
    (t - \tfrac{3}{5}\qj+\tfrac{4}{5}\qk + \eps(\tfrac{2}{5}\qj+\tfrac{3}{10}\qk))
    (t - \tfrac{3}{5}\qj+\tfrac{4}{5}\qk - \eps(\tfrac{2}{5}\qj + \tfrac{3}{10}\qk))
    (t+\tfrac{3}{5}\qj-\tfrac{4}{5}\qk),\\
    CD = (t + \qk)
    (t - \qk - \tfrac{1}{2}\eps\qj)
    (t - \qk + \tfrac{1}{2}\eps\qj).
  \end{gather*}
\end{example}

Above results state that existence of a motion polynomial factorization can be
guaranteed after multiplication with a real polynomial (which does not change
the underlying motion) or with a quaternion polynomial (which does not change
the trajectory of the origin). In \cite{gallet16} and \cite{li17} this was used
for the construction of linkages with a prescribed bounded rational trajectory
(Figure~\ref{fig:scissor-linkage}).

\subsection{Factorization of Unbounded Motion Polynomials}

If $C$ is an unbounded motion polynomial, existence of a factorization is not
guaranteed, not even after multiplication with a real polynomial $S \in \R[t]$
or a quaternion polynomial $D \in \H[t]$. Depending on the application one has
in mind, it might be possible to transform an unbounded motion polynomial into a
bounded motion polynomial. We may, for example substitute a rational expression
$A/B$ with $A$, $B \in \R[t]$ for the indeterminate $t$ in $C$ and try to factor
$B^{\deg C}C(A/B)$ instead. This amounts to a not necessarily invertible
re-parameterization of the motion. In particular, it is possible to parameterize
only one part of the original motion and transform $C$ to a bounded motion
polynomial.

However, there is a dense set of unbounded motions polynomials that admit a
factorization:

\begin{thm}
  \label{th:unbounded}
  If an unbounded motion polynomial $C = P + \eps Q \in \DH[t]$ with $P$, $Q \in
  \H[t]$ is such that all linear real factors of $\mrpf P$ have multiplicity
  one, there exists a real polynomial $D \in \R[t]$ such that $CD$ admits a
  factorization with linear motion polynomial factors.
\end{thm}

\begin{proof}
  We set $p' \coloneqq \mrpf P$ and denote by
  \begin{equation*}
    p = \prod_{i=1}^n(t-a_i),
    \quad
    a_1,a_2,\ldots,a_n \in \R.
  \end{equation*}
  the product of all monic
  linear real factors of $p'$. We then have $C\Cj{C} = p^2 U$ and $U \in \R[t]$
  has only irreducible quadratic real factors.

  We pick one linear factor of $p$, say $t-a_1$, and set $M \coloneqq
  (t-a_1)^2$. Because $\mrpf P$ has no linear real polynomial factor of
  multiplicity two, $M$ is not a pseudofactor of $C$ and Lemma~\ref{lem:3} can
  be applied to compute $h \coloneqq \czero(C,M)$. We now have $C =
  \tilde{C}(t-h)$ for some motion polynomial $\tilde{C}$ which is amenable to
  one further iteration of above construction (which is essentially one
  iteration of Algorithm~\ref{alg:factorization}). Treating all linear real
  factors of $p$ in like manner, we obtain a polynomial $H \in \DH[t]$ that
  admits a factorization with motion polynomial factors such that $C = C'H$ and
  $C' = P' + \eps Q'$ is bounded, that is, $\mrpf P'$ has no linear real factor.
  For bounded motion polynomials the statement is known to be true
  \cite[Theorem~1]{li18}. 
\end{proof}

There exist unbounded motion polynomials $C$ such that $CD$ does not admit a
factorization for all $D \in \H[t]$ (and in particular for real polynomials):

\begin{example}
  Consider the unbounded motion polynomial $C = (t-a_0)(t-a_1) + \eps\qi$ with
  $a_0 = a_1 = 0$ and a quaternion polynomial $D \in \H[t]$ with factorization
  $D = \prod_{i=2}^n (t-a_i)$ where $a_2$, $a_3$, \ldots, $a_m \in \H$. Then,
  the primal part of the product $CD$ has the factorization $\prod_{i=0}^n
  (t-a_i)$ and a suitable dual part exists if the system
  \begin{equation*}
    \qi D = \sum_{i=0}^n \Bigl(\prod_{j=0,\ j \neq i}^n(t-a_j)\Bigr) b_i
  \end{equation*}
  has a solution for $b_0$, $b_1$, \ldots, $b_n$. But this is not possible
  because the multiplicity of the factor $t$ on the right-hand side is always
  strictly larger than the multiplicity of this factor on the left-hand side.
\end{example}

\subsection{Factorization by Projection}
\label{sec:projection}

We conclude this text with a factorization technique applicable to non-motion
polynomials in $\DH$. Here, Algorithm~\ref{alg:factorization} fails already at
an early stage because the norm polynomial $\nu(C)$ is no longer real. More
generally, consider the Clifford algebra $\Clifford{p,q,1}$ and denote the basis
elements of $\R^n$ that square to $\pm 1$ by $e_1$, $e_2$, \ldots, $e_m$ where
$m = p + q$. There are $n = 2^m-1$ generators of $\Clifford{p,q,1}$ that are
products of above basis elements with non-zero square. We denote them by
$\qi_1$, $\qi_2$, \ldots, $\qi_n$ and we write $\eps$ for the generator that
squares to zero. Note that the real unit $1 = e_0$ has non-zero square as well.

Every element $c \in \Clifford{p,q,1}$ can be uniquely written as $c = a + b$
where $a \in \langle 1, \qi_1, \qi_2, \ldots, \qi_n \rangle$ and $b \in \langle
\eps,\qi_1\eps,\qi_2\eps,\ldots,\qi_n\eps \rangle$. In the context of dual
quaternions, $a$ is called the \emph{primal part} and $b$ is called the
\emph{dual part} and we use these notions here as well. A polynomial $C \in
\Clifford{p,q,1}$ has a unique representation as $C = A + B$ where $A$ is a
polynomial whose coefficients have zero dual part and $B$ is a polynomial whose
coefficients have zero primal part. We call $A$ and $B$, \emph{primal part} and
\emph{dual part,} respectively, of~$C$.

Assume now that the primal part of the monic polynomial $C$ admits a
factorization in $\Clifford{p,q,0}$, that is, $A = (t-a_1)(t-a_2) \cdots
(t-a_n)$ with $a_1$, $a_2$, \ldots, $a_n \in \Clifford{p,q,0}$. We make the
ansatz
\begin{equation}
  \label{eq:9}
  C = (t-a_1-b_1)(t-a_2-b_2) \cdots (t-a_n-b_n)
\end{equation}
with yet undetermined coefficients $b_1$, $b_2$, \ldots, $b_n$ of vanishing
primal part. Comparing coefficients on both sides of \eqref{eq:9} yields a
system of linear equations for the unknown real coefficients of $b_1$, $b_2$,
\ldots, $b_n$. The number of equations and the number of unknowns both equal
$(m+1)n$. Thus we can state:

\emph{If the primal part of a monic polynomial $C \in \Clifford{p,q,1}$ admits a
  factorization, a factorization of $C$ exists if the system of $(m+1)n$ linear
  equations in the same number of unknowns arising from comparing coefficients
  of \eqref{eq:9} has solutions.}

Generically, the solution to the linear system is unique but we already
encountered cases with infinitely many solutions or with no solution at all
(Examples~\ref{ex:no-factorization} and \ref{ex:5}). The algebra and geometry of
factorization of non-motion polynomials in $\DH[t]$ (and in particular a
kinematic interpretation) occurs in the theses \cite{scharler17,rad18} but
numerous open issues remain. In particular, sufficient criteria for existence of
factorizations, that is, solvability of the system of linear equations arising
from \eqref{eq:9}, would be desirable. While the factorization of motion
polynomials gives rise to a decomposition of rational motions into a sequence of
rotations, factorization of non-motion polynomials in $\DH[t]$ has in
interpretation as decomposition into so-called \emph{vertical Darboux motions}
\cite{rad18}.

\section*{Acknowledgment}

Daniel Scharler and Hans-Peter Schröcker were supported by the Austrian Science
Fund (FWF): P\;31061 (The Algebra of Motions in 3-Space). The authors gratefully
acknowledge useful comments and suggestions by anonymous reviewers that helped
to improve this text. All authors thank the Erwin Schr\"odinger International
Institute for Mathematics and Physics for its hospitality when completing the
last revision of this paper.
 
\bibliographystyle{elsarticle-num}

\end{document}